\documentclass[12pt,a4paper,twoside]{article}
\usepackage[a4paper,left=30mm,right=30mm,top=35mm,bottom=35mm]{geometry}
\usepackage{tikz}
\usetikzlibrary{positioning,arrows}

\usepackage{ifthen}
\usepackage{graphicx}
\usepackage{amsmath,amsfonts,euscript,amssymb,amsthm}
\usepackage{epsfig}
\usepackage{xcolor}

\newcommand{\R}{\mathbb{R}}
\newcommand{\N}{\mathbb{N}}

\newcommand{\Z}{\mathbb{Z}}

\newcommand{\eps}{\varepsilon}
\newcommand{\fhi}{\varphi}

\newcommand{\weak}{\rightharpoonup}
\newcommand{\weakto}{\rightharpoonup}
\newcommand{\weakstar}{\stackrel{\ast}{\rightharpoonup}}

\newcommand{\mean}{-\hspace{-1.15em}\int}

\newcommand{\del}{\partial}

\newcommand{\supp}{\mathrm{supp}}

\newcommand{\one}{\mathrm{\bf 1}}

\def\calK{\mathcal{K}}
\def\calL{\mathcal{L}}
\def\calM{\mathcal{M}}

\def\calL{\mathcal{L}}

\newtheorem{theorem}{Theorem}[section]

\newtheorem{lemma}[theorem]{Lemma}
\newtheorem{proposition}[theorem]{Proposition}

\newtheorem{assumption}[theorem]{Assumption}

\newtheorem{remark}[theorem]{Remark}

\numberwithin{equation}{section}

\begin{document}

\thispagestyle{empty}
\begin{center}
  ~\vskip7mm {\Large\bf Sound absorption by perforated walls\\[3mm] along
    boundaries
  }\\[8mm]

  {\large Patrizia Donato\footnotemark[1],
    Agnes Lamacz\footnotemark[2], and Ben Schweizer\footnotemark[3]}\\[5mm]


\end{center}

\footnotetext[1]{Laboratoire de Math\'ematiques, Avenue de
  l'Universit\'e BP.12, F-76801 Saint-\'Etienne-du-Rouvray, \tt
  patrizia.donato@univ-rouen.fr}

\footnotetext[2]{Fakult\"at f\"ur Mathematik, U Duisburg-Essen,
  Thea-Leymann-Stra\ss e 9, D-45127 Essen, \tt agnes.lamacz@uni-due.de}

\footnotetext[3]{Fakult\"at f\"ur Mathematik, TU Dortmund,
  Vogelspothsweg 87, D-44227 Dortmund, \tt ben.schweizer@tu-dortmund.de}

\pagestyle{myheadings} \markboth{Sound absorption by perforated walls
  along boundaries}{A. Lamacz and B. Schweizer}

\begin{center}
   \begin{minipage}[c]{0.85\textwidth}
     {\bf Abstract:} We analyze the Helmholtz equation in a complex
     domain. A sound absorbing structure at a part of the boundary
     is modelled by a periodic geometry with periodicity $\eps>0$.  A
     resonator volume of thickness $\eps$ is connected with thin
     channels (opening $\eps^3$) with the main part of the macroscopic
     domain. For this problem with three different scales we analyze
     solutions in the limit $\eps\to 0$ and find that the effective
     system can
     describe sound absorption.\\[-1mm]


   \end{minipage}\\[2mm]
\end{center}


\section{Introduction}

We are interested in the mathematical analysis of a sound absorbing
structure, e.g., along the wall of a room. The sound absorber consists
of a combination of small-scale structures. For the simplest setting
one should think of a wooden plate that is attached to the wall. The
plate is attached in such a way that a thin gap remains between plate
and wall. To create the sound absorption effect, little holes are
drilled in the wood to connect the room with the thin volume behind
the plate.

In order to analyze the effects of such a structure, we define a
geometry with different small scales: The wood is modelled by a layer
of thickness $\eps>0$, the (resonator) volume behind the wood has also
a thickness of order $\eps$, the holes are distributed periodically
with periodicity $\eps$. The width of the holes is assumed to be of
order $\eps^3$; this is the scaling in which a nontrivial limit
behavior is observed. We study the Helmholtz equation in the domain
that is filled with air, using homogeneous Neumann conditions along
all boundaries. Denoting solutions by $u^\eps$, we are interested in
the behavior of $u^\eps$ in the limit $\eps\to 0$. We find two
effective systems; they describe sound waves in the volume with the
sound absorbing structure.

We derive {\em two} limit systems, since the lowest order
approximation is trivial. At order $\eps^0$, the limit problem
coincides with the original Helmholtz problem: The small structures
along the boundary have no effect. In this sense, the complex geometry
can only lead to an effect of order $\eps$. We derive the effective
system for this $O(\eps)$ deviation in Theorem \ref {thm:mainTheorem}
below. Due to $L^2$-unboundedness of relevant functions, the proof is
performed with $L^1$-based function spaces and limit measures for
pressure and flux quantities.

The interesting question from the modelling perspective is: Why can
the $O(\eps)$ deviation be relevant for sound absorption? We see the
answer in the effective equation of Theorem \ref {thm:mainTheorem}:
The effective system contains the quantity $(\alpha/(LV)) - \omega^2$,
where $\alpha, L, V$ are geometric quantities, and $\omega$ is the
frequency. When the frequency is near to $\sqrt{\alpha/(LV)}$,
resonance occurs and the solutions of the $O(\eps)$-system can be very
large.  When they are of the same order as the inverse periodicity
(i.e.: $\eps^{-1}$), then the sound absorber can have a relevant
effect. This is discussed towards the end of this introduction.

\paragraph{\bf Geometry.}

We next describe the domain $\Omega_\eps$. It consists of a volume
$\Omega_0$ and some small scale structures that are attached to one
part of the boundary of $\Omega_0$. To keep the setting simple, we
assume that $\Omega_0$ is a rectangle in $\R^2$. With the two positive
parameters $a,b>0$ we denote by the interval $I := (0,a)$ the range of
the horizontal coordinate $x_1$. The limit domain is
$$\Omega_0 := (0,a) \times (-b,0) = I \times (-b,0)\,,$$
the upper boundary of $\Omega_0$ is the set
$\Gamma_0 := I\times \{0\}$.  By slight abuse of notation we will
identify functions on $I$ with functions on $\Gamma_0$.

Attached to $\Gamma_0$ is the resonator volume, which is connected
with many thin channels to the volume $\Omega_0$.  The channels are
distributed periodically with a spacing $\eps>0$; our analysis is
concerned with the limit $\eps\to 0$.  We denote by $L>0$ and $V>0$
the relative length of the channels and the relative thickness of the
resonator volume, respectively. The parameter $\alpha>0$ denotes a
relative width of the channels.

\begin{figure}[ht]
  \centering
  \begin{tikzpicture}[scale = 0.62]
    \draw[thick] (0, -6) to (10,-6) to (10,2.5) to (0,2.5) to (0,-6) circle;

    \node[] at (0.8, 1.75) {$S_\eps$};
    \node[] at (-0.8, 0.5) {$C_\eps$};
    \node[] at (0.8, -5) {$\Omega_0$};

    \draw[->] (-1, 0.0) -- (12.5, 0.0);
    \node[] at (12.5, -.52) {$x_1$};
    \draw[->] (0.0, -6.5) -- (0.0, 3.8);
    \node[] at (0.55, 3.9) {$x_2$};

    \draw[<->] (10.7, 0.0) -- (10.7, 1.0);
    \node[] at (11.4, 0.5) {$L \eps$};
    
    \draw[<->] (10.7, 2.5) -- (10.7, 1.0);
    \node[] at (11.4, 1.75) {$V \eps$};

    \foreach \x in {0,...,19}
    {
      \tikzset{shift={(0.5*\x, 0)}}
      \draw[] (0.2, 0.0) to (0.5, 0.0) to (0.5,1.0)
      to (0.2,1.0) to (0.2, 0.0) circle;
    }
    
    \end{tikzpicture}
    \caption{\small The geometry. The complex domain $\Omega_\eps$ is
      given as the union of a limit domain $\Omega_0$ (the domain
      below the $x_1$-axis), the set of channels $C_\eps$, and the
      strip $S_\eps$ above the channels. The length of the channels is
      $L\eps$, the width of the strip $S_\eps$ is $V\eps$. The
      channels are distributed with periodicity $\eps$, the width of
      the channels is $\alpha \eps^3$.}
  \label{fig:num1}
\end{figure}
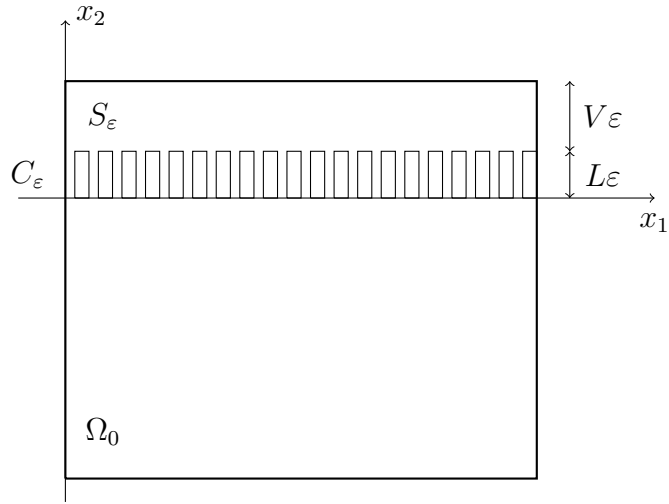

The domain $\Omega_\eps$ is constructed as the union of three sets as
described below (see Figure 1).  For simplicity we always assume
$a/\eps\in \N$. The resonator strip and the channels are
\begin{align}
  S_\eps &:= I\times (L\eps, (L+V)\eps)\,,\\
  C_\eps &:= \bigcup_{k = 0}^{(a/\eps) - 1} (k\eps, k\eps +
           \alpha\eps^3)\times [0,L\eps]\,,
\end{align}
and the domain is defined as
\begin{equation}
  \label{eq:Omega-eps}
  \Omega_\eps := \Omega_0\cup S_\eps\cup C_\eps\,.
\end{equation}
The upper boundary of $\Omega_\eps$ is
$\Gamma_\eps := I\times \{ (L+V) \eps\}$. We emphasize the fact that
three scales are involved, since the channels have the width
$\alpha\eps^3$. The total volume of the channels is of the order
(length $\times$ width $\times$ number)
$|C_\eps| \sim \eps\, \eps^3\cdot \eps^{-1} = \eps^3$.

\paragraph{\bf Main results.}
We are interested in the limit behavior of a sequence $u^\eps$
satisfying the Helmholtz equation
\begin{equation}\label{eq:HH}
  \begin{split}
    -\Delta u^\eps - \omega^2 u^\eps &= f\qquad\text{ in } \Omega_\eps\,,\\
    \del_n u^\eps &= 0\qquad\text{ on } \del\Omega_\eps\,.
  \end{split}
\end{equation}
Throughout, we assume that the frequency $\omega>0$ and the right hand
side $f\in L^2(\R^2)$ are given. To simplify calculations, we assume
that $f$ has support in $\Omega_0$.

\smallskip We first provide the following theorem in order to stress
that the limit system for \eqref {eq:HH} is trivial.

\begin{theorem}[Trivial limit equation]\label{thm:0}
  Let $(u^\eps)_{\eps>0}$ be a sequence of solutions to \eqref {eq:HH}
  for some sequence $\eps\to 0$. We assume that
  $u^\eps\in H^1(\Omega_\eps)$ is bounded and that a weak limit
  $u\in H^1(\Omega_0)$ exists,
  \begin{equation}
    \label{eq:u-converg-ass}
    u^\eps|_{\Omega_0} \weakto u \quad\text{in } H^1(\Omega_0)\,.
  \end{equation}
  Then $u$ solves the trivial limit problem
  \begin{equation}
    \begin{split}
      -\Delta u - \omega^2 u &= f\qquad\text{ in } \Omega_0\,,\\
      \del_n u &= 0\qquad\text{ on } \del\Omega_0\,.
    \end{split}\label{eq:HH-trivial}
  \end{equation}
\end{theorem}

\begin{proof}
  Let $\varphi\in C^1(\R^2)$ be an arbitrary test function.  In the
  following calculation, we use first the volume estimate
  $|\Omega_\eps\setminus\Omega_0| = O(\eps)$, then the weak form of
  \eqref{eq:HH}, and finally decompose the integral and exploit the
  boundedness of the sequence $u^\eps$:
  \begin{align*}
    \int_{\Omega_0}f\varphi
    &\leftarrow \int_{\Omega_\eps}f\varphi
    = \int_{\Omega_\eps}\nabla u^\eps\cdot\nabla \varphi
      -\omega^2\int_{\Omega_\eps}u^\eps\varphi\\
    &=\int_{\Omega_0}\nabla u^\eps\cdot\nabla \varphi
       +\int_{\Omega_\eps\setminus\Omega_0}\nabla u^\eps\cdot\nabla \varphi
       -\omega^2\int_{\Omega_0}u^\eps\varphi
       -\omega^2\int_{\Omega_\eps\setminus\Omega_0}u^\eps\varphi\\
    &\to
      \int_{\Omega_0}\nabla u\cdot\nabla \varphi-\omega^2\int_{\Omega_0}u\,\varphi
  \end{align*}
  as $\eps\to 0$.  We thus obtained the weak form of
  \eqref{eq:HH-trivial}.
\end{proof}

The interesting effect in the behavior of solutions becomes visible in
the next order in $\eps$.  We define two new functions. The first
encodes the averages of $u^\eps$ with respect to the variable $x_2$ in
the resonator strip $S_\eps$,
\begin{equation}
  \label{eq:v-eps}
  v^\eps : I \to \R,\quad
  x_1 \mapsto \frac1{\eps V} \int_{\eps L}^{\eps(L+V)} u^\eps(x_1, x_2)\, dx_2\,,
\end{equation}
and the second denotes the corrector from the trivial limit,
\begin{equation}
  \label{eq:w-eps}
  w^\eps : \Omega_0 \to \R,\quad
  w^\eps := \frac{u^\eps - u}{\eps}\,.
\end{equation}

We work with the following assumption.

\begin{assumption}\label{ass:v-w}
  For some $v\in H^1(I)$ there holds $v^\eps\weakto v$ in
  $L^2(I)$. Moreover, the sequence $w^\eps$ is bounded in
  $W^{1,1}(\Omega_0)$ and, for some $w\in W^{1,1}(\Omega_0)$,
  $w^\eps\weakto w$ weak-$*$ in $BV(\bar\Omega_0)$.  The sequence
  $u^\eps\in H^1(\Omega_\eps)$ is bounded and the vertical derivative
  of $u^\eps$ satisfies the following boundedness in the channels: For
  some constant $C>0$, that does not depend on $\eps$, holds
  \begin{align}
  \label{eq:assdel2ueps}
  &\frac{1}{\eps^2}\int_{C_\eps}|\partial_2 u^\eps|\leq C\,.
  \end{align}
\end{assumption}

The weak-$*$ convergence of $w^\eps\to w$ in $BV(\bar\Omega_0)$ is
equivalent to: $w^\eps\to w$ in $L^1(\Omega_0)$ and
$\int_{\Omega_0}\nabla w^\eps\cdot \phi\to\int_{\Omega_0}\nabla
w\cdot\phi$ for all $\phi\in C(\bar\Omega_0;\R^2)$.

\smallskip
For the heuristics of Assumption \ref{ass:v-w} we refer to Section
\ref {ssec.heuristics} below.

\begin{remark}\label{rem:lower-reg-w}
  In what follows it would be sufficient to assume that $\nabla w$ is
  a measure, which is the natural assumption in the context of weak
  $BV$-convergence.  For the sake of simplicity of notation we stick
  to the stronger assumption $w\in W^{1,1}(\Omega_0)$.
\end{remark}

We are now in a position to formulate the
main result of this article. It determines the limit equation for the
function $w$. By definition of $w$, the solution $u^\eps$ has the
expansion $u^\eps \approx u + \eps w$.

\begin{theorem}[Equations for the corrector]
  \label{thm:mainTheorem}
  Let $u^\eps$ and $u$ be as in Theorem \ref {thm:0}. Let $v^\eps$ and
  $w^\eps$ be as in \eqref {eq:v-eps} and \eqref {eq:w-eps}. Let
  Assumption \ref {ass:v-w} hold with limits $v$ and $w$. Then the
  equation for $w$ is
  \begin{equation}
    \begin{split}
      -\Delta w - \omega^2 w &= 0\quad\qquad\qquad\qquad\text{ in } \Omega_0\,,\\
      \del_n w &= V\, (\del^2_1+\omega^2) v
      \qquad \text{ on } \Gamma_0\,,\\
       \del_n w &= 0
      \qquad\qquad\qquad\quad\text{ on } \del\Omega_0\setminus\Gamma_0\,,
    \end{split}\label{eq:HH-w}
  \end{equation}
  and the equation for $v$ is
  \begin{equation}
    \label{eq:limiteq-v}
    \left(-\del_1^2 + \left(\frac{\alpha}{LV}-\omega^2\right)\right) v
    = \frac{\alpha}{LV} u|_{\Gamma_0} \,.
  \end{equation}

  The function $v$ has the regularity $v\in W^{2,1}(I)$. System \eqref
  {eq:HH-w}--\eqref {eq:limiteq-v} has to be understood in the weak
  sense: For every $\fhi\in C^1(\bar\Omega_0)$ holds
  \begin{equation}\label{eq:main-w-weak}
    \int_{\Omega_0} \nabla w\cdot \nabla\fhi
    - \omega^2 \int_{\Omega_0}  w\, \fhi
      =  - V \int_{\Gamma_0} (\del_1 v\, \del_1 \fhi
      - \omega^2 v\, \fhi ) \,,
  \end{equation}
  and for every $\psi\in C^1(\bar\Gamma_0)$ holds
  \begin{equation}
    \label{eq:limiteq-v-weak}
    \int_{\Gamma_0} \del_1 v\, \del_1 \psi
    + \int_{\Gamma_0} \left(\frac{\alpha}{LV}-\omega^2\right) v \,\psi
    = \int_{\Gamma_0} \frac{\alpha}{LV} u \,\psi \,.
  \end{equation}
  We note that \eqref {eq:limiteq-v-weak} encodes not only \eqref
  {eq:limiteq-v}, but additionally the homogeneous Neumann boundary
  condition $\del_1 v = 0$ at $\del I$.
\end{theorem}

We have formulated the limiting system in a form that shows the
existence and uniqueness of solutions for almost all frequencies
$\omega$. The limit problem for $u$ has a unique solution if $\omega$
is not an eigenvalue of the Neumann Laplace operator on
$\Omega_0$. Given $u\in H^1(\Omega_0)$ and its trace
$u|_{\Gamma_0} \in H^{1/2}(\Gamma_0)$, equation \eqref {eq:limiteq-v}
with Neumann boundary conditions $\del_1 v = 0$ at $\del I$ can be
solved for $v\in H^2(I)$. Finally, assuming again that $\omega$ is not
an eigenvalue of the Neumann Laplace operator on $\Omega_0$, we can
solve system \eqref {eq:HH-w} for $w\in H^1(\Omega_0)$. This line of
argument yields not only existence, but also uniqueness of solutions
with $w$ of class $H^1$.

We note that we required less regularity on $w$ in Assumption
\ref{ass:v-w}. Our results imply that, if the limit has the additional
regularity $w\in H^1(\Omega_0)$, then it necessarily coincides with
the unique $H^1$-solution of the limit system (for $\omega$ not an
eigenvalue of the Neumann Laplace operator of $\Omega_0$).

The limit equation \eqref{eq:limiteq-v} can be re-written as
\begin{equation}\label{eq:equiv-boundary-eq}
  (\del^2_1+\omega^2)v = \frac{\alpha}{LV}(v-u)\quad\text{on }\Gamma_0\,.
\end{equation}
The boundary condition for $w$ along $\Gamma_0$ can therefore be
expressed as $\del_n w = \frac{\alpha}{L} (v-u)$ and equation \eqref
{eq:main-w-weak} can be written as
\begin{equation*}
  \int_{\Omega_0}\nabla w\cdot\nabla\varphi -\omega^2\int_{\Omega_0}w\varphi
  = \frac{\alpha}{L} \int_{\Gamma_0} (v-u) \varphi\,.
\end{equation*}
The derivation of \eqref {eq:HH-w} is actually not difficult, we
present the proof in Proposition \ref{prop:A}. The connections between
$u$ and $v$ are more involved, we derive two relations in Propositions
\ref{prop:geometricflowrule} and \ref{prop:B}.  Theorem
\ref{thm:mainTheorem} is proved after Proposition \ref{prop:B}.

\paragraph{\bf Interpretation of the main result.}

As stressed before, the limit solution $u$ is not affected by the
small scale structures along the boundary.

Let us study the limit equation \eqref{eq:limiteq-v}. The
function $v$ depends only on the horizontal coordinate $x_1$. Let us
consider solutions of the form $v(x_1) = v_0 \sin(k x_1)$ and
$u(x_1, 0) = u_0 \sin(k x_1)$ for some real parameters
$v_0, u_0 \in \R$.  Equation \eqref{eq:limiteq-v} then reads
\begin{equation*}
  \left( k^2  + \frac{\alpha}{LV} - \omega^2\right) v_0
  = \frac{\alpha}{LV} u_0\,.
\end{equation*}
This relation implies that, for resonant frequencies $\omega$, the
factor $v_0$ can be much larger than the factor $u_0$. For small
horizontal wave numbers $k$, this occurs when $\omega$ is close to the
Helmholtz resonator frequency $\omega_H := \sqrt{\alpha/(LV)}$.

When all the functions $w, u$, and $v$ have the dependence
$\sin(kx_1)$ on $x_1$, then the problem for $w$ is a homogeneous
Helmholtz problem with the upper boundary condition
\begin{align*}
  \del_n w &= \frac{\alpha}{L} (v-u)=\frac{\alpha}{L}(v_0-u_0)\sin(k x_1)\\
           &=\frac{\alpha}{L}\left[\frac{\alpha}{LV}
             \left(k^2+\frac{\alpha}{LV}-\omega^2\right)^{-1} - 1
             \right]u_0\sin(kx_1)\,.
\end{align*}
The factor in squared brackets can be large due to resonance (small
denominator). This results in large values of the function $w$. In the
reconstruction of $u^\eps$ we obtain $u^\eps \approx u + \eps w$, and
the correction has the order
$\eps \| w\| = O(\eps ( k^2 - \omega^2 + \frac{\alpha}{LV}
)^{-1})$. Due to the resonance, this can constitute a visible (or,
better: audable) contribution even for small periodicity length
$\eps>0$.

\paragraph{\bf Literature.}

Some of the first mathematical results in the field of homogenization
regarded the derivation of limit equations for domains that are
periodically perforated, see, e.g., \cite{Cioranescu1997}. Quickly,
the interest shifted also to geometries where the perforations are
along lower dimensional manifolds, we refer to \cite {MR1655534,
  MR658023} for two early contributions. The periodic unfolding method
was adapted to this kind of problems, see \cite {MR2401689}.  For the
problem in the context of fluid mechanics, see \cite {MR884812}.

As a ``natural scaling'' we regard the setting where the periodicity
is $\eps>0$, and the typical size of the obstacles is also $\eps$ (in
every direction). This scaling was also considered in the papers \cite
{MR3324482, MR2935369, MR3109435, DHS, Schweizer-Neumann-2018}.  The
aim of these papers is to provide a thorough analysis of the Neumann
problem, for which no effects of order $\eps^0$ are induced by the
geometry. In order to derive limit equations one has to analyze higher
order effects. Progress was possible in \cite {Schweizer-Neumann-2018}
with the consequent use of $W^{1,1}$-spaces: the expansion of the
solution has natural bounds in the corresponding norms.

We emphasize that, in the natural scaling, where periodicity, width,
and the length of the channels are all of order $\eps$, no resonances
can occur. In such a setting, one can only expect that deviations from
the trivial limit solution $u$ are of order $\eps$.

We note that another scaling is used, e.g., in
\cite{Plasmonic-waves-2013, Neuss-Radu-Jaeger, MR3808152}: Here, a
structure of finite width is analyzed. For a periodicity $\eps>0$ and
a diameter of the channels of order $\eps$, the length of the channels
does not tend to $0$ as $\eps\to 0$. This scaling allows for
resonances in the longitudinal direction of the channels.  Yet another
setting of the geometry was used, e.g., in \cite {MR2151798}: One
considers ``perforations'' in the boundary or in an interface of lower
dimension. The resulting system has the character of an oscillatory
boundary condition, we mention \cite {MR2739003} as a contribution in
this vast field.

The combination of two different small length scales in the obstacles
can create resonant structures. This is well-known for the Helmholtz
resonator and it was used for an analysis of spectral properties in
\cite {Schweizer-Helmholtz-resonator}. Using the small Helmholtz
resonator as a building block, one can create resonant bulk materials,
see \cite {Lamacz-Schweizer-manyres-2017}. In that work, the
resonators are distributed in the whole volume and not only along the
boundary.  For an overview regarding resonances and homogenization in
this spirit, we mention \cite {MR3591457}.

\section{Preliminaries and proof of \eqref {eq:main-w-weak}}

\subsection{Expected orders of different quantities}
\label{ssec.heuristics}

It might be surprising that we work with $L^1$-based spaces. The
choice of the function space is important. In fact, we claim that
working only in $L^2$-based function spaces is not adequate in the
problem at hand. We note that a similar observation was made in \cite
{Schweizer-Neumann-2018}.

Let us discuss heuristically the behavior of solutions.  We expect
that $u^\eps$ has values of order $\eps^0 = 1$ everywhere, in the
domain $\Omega_0$ and in the resonator strip $S_\eps$.

Since the channels are thin, there is is only a weak connection
between the volume $\Omega_0$ and the strip $S_\eps$. There is no
reason why the values of $u^\eps$ at both ends of the channel should
be close. We can therefore expect that also the difference $v- u$ is
of order $1$. As a result, since the length of each channel is of
order $\eps$, the derivative $\del_2 u^\eps$ should be of order
$\eps^{-1}$ in the channels.

With respect to \eqref {eq:assdel2ueps} we recall that the total
volume of the channels is of the order $|C_\eps| \sim \eps^3$. We can
therefore expect that the quantity in \eqref {eq:assdel2ueps},
$\eps^{-2} \int_{C_\eps}|\partial_2 u^\eps|$, is bounded.

We note that the boundedness of $\nabla u^\eps \in L^2(\Omega_\eps)$
implies the following property of horizontal derivatives:
\begin{equation}\label{eq:assdel1ueps}
  \frac{1}{\eps}\int_{C_\eps}|\partial_1 u^\eps|
  \le \frac{1}{\eps} \| \nabla u^\eps \|_{L^2(\Omega_\eps)} |C_\eps|^{1/2}
  \le C \eps^{3/2 - 1} \to 0\,.
\end{equation} 

We include the warning that $L^2$-spaces are not adequate for this
problem. We expect
$\eps^{-2} \int_{C_\eps}|\partial_2 u^\eps|^2 \sim \eps^{-2} |C_\eps|
\eps^{-2} \sim \eps^{-1} \to \infty$.  In particular, we do not expect
that $L^2$-norms of $\nabla w^\eps$ are bounded. This is why we work
in the $L^1$-family of norms.

\subsection{Proof of \eqref {eq:main-w-weak}}

In this subsection we derive that the corrector $w$ satisfies
\eqref{eq:HH-w}. More precisely, we derive the weak form \eqref
{eq:main-w-weak}.

\begin{proposition}
  \label{prop:A}
  Let the sequence $u^\eps$ be as in Theorem
  \ref{thm:mainTheorem}. Then the limit function $w$ satisfies the
  effective equation \eqref {eq:main-w-weak}.
\end{proposition}

\begin{proof}
  Let $\fhi\in C^1(\R^2)$ by an arbitrary test function.  Decomposing
  the integral over $\Omega_\eps$ into intgrals over $\Omega_0$ and
  $S_\eps\cup C_\eps$, the equation for $u^\eps$ reads
  \begin{align*}
    \int_{\Omega_0}\nabla u^\eps\cdot\nabla\fhi-\omega^2\int_{\Omega_0}u^\eps\fhi
    + \int_{S_\eps\cup C_\eps}\nabla u^\eps\cdot\nabla\fhi
    -\omega^2\int_{S_\eps\cup C_\eps}u^\eps\fhi
    = \int_{\Omega_\eps}f\fhi=\int_{\Omega_0}f\fhi\,.
  \end{align*}
  On the other hand, the equation for $u$ provides
  \begin{align*}
    \int_{\Omega_0}\nabla u\cdot\nabla\fhi-\omega^2\int_{\Omega_0}u\fhi 
   =\int_{\Omega_0}f\fhi\,.
  \end{align*}
  We subtract the two equations, divide by $\eps$, and insert the
  definition $w^\eps = (u^\eps - u)/\eps$ to obtain
  \begin{equation}
  \label{eq:weakformweps}
    \int_{\Omega_0} \nabla w^\eps\cdot \nabla\fhi
    - \omega^2 \int_{\Omega_0}  w^\eps\, \fhi
    = - \frac1{\eps} \left\{ \int_{S_\eps \cup C_\eps} \nabla u^\eps\cdot \nabla\fhi
      - \omega^2 \int_{S_\eps \cup C_\eps}  u^\eps\, \fhi \right\}\,.
  \end{equation}

  We consider test functions $\fhi$ that have the regularity
  $\fhi|_{\bar\Omega_0} \in C^1(\bar\Omega_0)$, assume that they are
  independent of $x_2$ for $x_2\geq 0$ and that they satisfy
  $\del_1\fhi = 0$ in the set
  $S_\eps\cap \{x_1 < \delta \text{ or } x_1 > a - \delta\}$ for some
  $\delta>0$.

  Assumption \ref{ass:v-w} on $w^\eps$ implies that the left hand side
  of \eqref{eq:weakformweps} converges, as $\eps\to 0$,
  \begin{equation*}
    \int_{\Omega_0} \nabla w^\eps\cdot \nabla\fhi
    - \omega^2 \int_{\Omega_0}  w^\eps\, \fhi
    \to\int_{\Omega_0} \nabla w\cdot \nabla\fhi
    - \omega^2 \int_{\Omega_0}  w\, \fhi\,.
  \end{equation*}
  
  We can use Assumption \ref{ass:v-w} also to calculate the right hand
  side of \eqref{eq:weakformweps}. All integrals over $C_\eps$ vanish
  in the limit $\eps\to 0$ because of boundedness of $u^\eps$ and
  $\nabla u^\eps$ in $L^2(\Omega_\eps)$. For one of the remaining two
  integrals, we use $\del_2 \fhi =0$ in $S_\eps$ and an integration by
  parts to find
  \begin{align*}
    -\frac1{\eps} \int_{S_\eps} \nabla u^\eps\, \nabla\fhi
    = -\frac1{\eps} \int_{S_\eps} \del_1 u^\eps\, \del_1\fhi  
    =\frac{1}{\eps}\int_{S_\eps}u^\eps\,\del^2_1\fhi
    \to V \int_I v\, \del^2_1 \fhi
    = - V \int_I \del_1 v\, \del_1 \fhi\,.
  \end{align*}
  The last integral satisfies
  \begin{align*}
    \frac1{\eps} \int_{S_\eps}  u^\eps \fhi \to V \int_I  v  \fhi\,,
  \end{align*}
  as $\eps\to 0$.  Combining these limits, we arrive at
  \begin{align*}
    &\int_{\Omega_0} \nabla w\cdot \nabla\fhi
    - \omega^2 \int_{\Omega_0}  w\, \fhi
      =  V \int_I (-\del_1 v(x_1)\, \del_1 \fhi(x_1,0)
      + \omega^2 v(x_1) \fhi(x_1,0))\,dx_1 \,.
  \end{align*}
  By density of functions $\fhi$ as above, this relation holds for all
  $\fhi\in C^1(\bar\Omega_0)$. We have obtained \eqref
  {eq:main-w-weak}.
\end{proof}

\subsection{A flux quantity}

Relation \eqref{eq:limiteq-v} between $u$ and $v$ requires much more
involved arguments. We introduce an additional quantity, the vertical
flux $j^\eps$ and its limit $j_*$. We set
\begin{equation}
  \label{eq:currentdef}
  j^\eps(x) := \frac{1}{L\eps^2}\, \del_2 u^\eps(x)\, \one_{C_\eps}(x)\,,
\end{equation}
where $\one_{C_\eps}$ is the characteristic function of the channels,
$\one_{C_\eps}(x) = 1$ for $x\in C_\eps$ and $\one_{C_\eps}(x) = 0$
for $x\not\in C_\eps$. The quantity $j^\eps$ measures, in a rescaled
fashion, the $x_2$-derivative of $u^\eps$ in the channels.

\begin{lemma}
  \label{lem:currentintroduction} 
  Let $u^\eps$ be a sequence as in Theorem \ref {thm:mainTheorem} and
  let $j^\eps$ be as in \eqref{eq:currentdef}. Then there exists a
  subsequence $\eps\to 0$ and a Radon measure
  $j_*\in\mathcal{M}(\R^2)$ with $\supp(j_*)\subset \bar\Gamma_0$ such
  that
  \begin{equation}
    \label{eq:j-eps-star}
    j^\eps\weakstar j_* 
  \end{equation}
  in the sense of Radon measures.
\end{lemma}

\begin{proof}
  By Assumption \eqref{eq:assdel2ueps} on $\del_2 u^\eps$, the current
  $j^\eps$ is uniformly bounded in $L^1(\R^2)$:
  \begin{equation*}
    \|j^\eps\|_{L^1(\R^2)}=\frac{1}{L\eps^2}\int_{C_\eps}|\del_2 u^\eps|
    \leq \frac{C}{L}\,.
  \end{equation*}
  With the two-dimensional Lebesgue measure $\calL^2$, we can consider
  $j^\eps\, \calL^2$ as a bounded family of measures. The weak star
  compactness of Radon measures implies the existence of a subsequence
  and of a limit measure $j_*\in\mathcal{M}(\R^2)$ with \eqref
  {eq:j-eps-star}.  The measure $j_*$ is concentrated on $\bar\Gamma_0$
  since the measures $j^\eps\, \calL^2$ are supported in the channels
  $C_\eps$, hence in an $\eps$-neighborhood on $\Gamma_0$.
\end{proof}

We will use Lemma \ref {lem:currentintroduction} as follows: For every
function $\fhi\in C(\R^2)$ there holds, as $\eps\to 0$,
\begin{equation}
  \frac{1}{L\eps^2}\int_{C_\eps}\del_2 u^\eps(x)\fhi(x)\,dx
  \to \int_{\bar\Gamma_0}\fhi(x)\,dj_*(x)\,. 
\end{equation}

As a preparation of one of the subsequent proofs, we note that the
arguments of Lemma \ref {lem:currentintroduction} can be repeated for
the absolute values of $j^\eps$: We consider $J^\eps := |
j^\eps|$. The measures $J^\eps \calL^2$ are a bounded family of Radon
measures. Along a subsequence $\eps\to 0$ we can therefore assume, for
some limit Radon measure $J_* \in \calM(\R^2)$ with support in
$\bar\Gamma_0$, that $J^\eps\weakstar J_*$ in the sense of Radon measures.

\section{Relations between $u$ and $v$}
\label{sec:geomflowrule}

In this section we obtain equation \eqref{eq:limiteq-v} for $u$ and
$v$. It is obtained from two other relations that involve $u$, $v$,
and the flux quantity $j_*$: The geometric flow rule \eqref
{eq:geometricflowrule} and the mass conservation \eqref
{eq:densityjstar1}. Upon eliminating the flux $j_*$, we obtain
\eqref{eq:limiteq-v}.

The first of these two new relations is the geometric flow rule and is
shown in Proposition \ref{prop:geometricflowrule}.  This geometric
rule can be perceived as follows: When $u^\eps$ has the typical value
$v$ at the upper end of the channel and the typical value $u$ at the
lower end of the channel, then the derivative has the typical value
$\del_2 u^\eps \sim (v-u)/(L\eps)$. For the integral of $j^\eps$ over
a single channel (with length $L\eps$ and width $\alpha \eps^3$) we
therefore expect to obtain
$L\eps\, \alpha \eps^3 / (L\eps^2)\cdot (v-u)/(L\eps) = (\alpha\eps/L)
(v-u)$. The factor $\eps$ denotes the periodicity. We therefore expect
a relation of the form $j_* = (\alpha/L) (v-u)$. The argument is made
precise in the following Proposition.

\begin{proposition}[Geometric flow rule]
  \label{prop:geometricflowrule}
  Let $j_*$ be as in Lemma \ref{lem:currentintroduction}, $v$ and $u$
  as in Assumption \ref{ass:v-w} and Theorem \ref{thm:0}.  Then there
  exists a density function $j\in L^1(I)$ such that
  \begin{equation}
    \label{eq:reljstarv}
    j_*(x)=j(x_1)\,\mathcal{H}^1|_{\Gamma_0}\,.
  \end{equation}
  The density satisfies
  \begin{equation}
    \label{eq:geometricflowrule}
    j(x_1)=\frac{\alpha}{L}(v(x_1)-u(x_1,0))\,.
  \end{equation}
\end{proposition}

\begin{proof}
  Once \eqref{eq:reljstarv}--\eqref{eq:geometricflowrule} are shown,
  the $L^1$-regularity of $j$ follows directly from the fact that
  $u(\cdot,0)$ and $v$ are of class $L^1(I)$.  We only have to prove
  \eqref{eq:reljstarv}--\eqref{eq:geometricflowrule}.

  Let $[c,d]\subset [0,a] = \bar I$ be an interval. Since $j_*$ is
  supported on $\bar\Gamma_0$, Proposition
  \ref{prop:geometricflowrule} is proved as soon as we can show that the
  limit measure $j_*$ satisfies
  \begin{equation}
    \label{eq:toshowgeomflowrule}
    \int_{[c,d]\times\{0\}}\,dj_*=\frac{\alpha}{L}\int_c^d (v(x_1)-u(x_1,0))\,dx_1\,.
  \end{equation}

  We will use the following function $\theta^\eps$ with large
  gradients:
  \begin{align}
    \label{eq:functionlargegradient}
    \theta^\eps(x_2):=
    \begin{cases}
      0\quad&\text{for } x_2\leq 0\,,\\
      x_2/(\eps L)\quad&\text{for } 0<x_2<\eps L\,,\\
      1\quad&\text{for } x_2\geq \eps L\,.
    \end{cases}
  \end{align}

  We want to use a localization function $\psi_\eps: [0,a]\to \R$. As
  a test function we then consider
  $\fhi^\eps(x_1,x_2) := \psi_\eps(x_1)\theta^\eps(x_2)$. The proof of
  \eqref {eq:toshowgeomflowrule} consists in calculating the quantity
  \begin{equation}
    \label{eq:defBeps}
    B_\eps:=\frac1\eps\int_{C_\eps}\nabla u^\eps\cdot \nabla\fhi^\eps
  \end{equation}
  in two different ways.

  As localization function $\psi_\eps$ we cannot use the
  characteristic function $\chi_{[c,d]} : [0,a] \to \{0,1\}$ of the
  interval $[c,d]$, since the jumps of this function can occur within
  a channel. We choose to consider all cells that touch the interval
  $[c,d]$: We define a set $\calK_\eps$ of indices as
  \begin{equation}
    \label{eq:defIeps}
    \calK_\eps :=
    \{k_1\in\Z\,|\, \eps k_1 \in [0,a-\eps]\text{ and }
    \left(k_1\eps,k_1\eps+\eps\right)\cap [c,d] \neq \emptyset\}\,.
  \end{equation}
  The number of elements of $\calK_\eps$ is of order
  $|\calK_\eps| = O(\eps^{-1})$.  We furthermore introduce
  $I^\eps_{c,d} := \bigcup_{k_1\in
    \calK_\eps}\left(k_1\eps, k_1\eps+\eps\right)$ and set
  \begin{equation}
    \psi_\eps(x_1) :=
    \begin{cases}
      1 \quad &\text{for }
      x_1 \in I^\eps_{c,d}\,,\\
      0  \quad &\text{else\,. }
    \end{cases}
  \end{equation}

  \smallskip \textit{First calculation of $B_\eps$.} We write $B_\eps$
  with the flux variable $j^\eps$ as
  \begin{align*}
    B_\eps
    &=\frac1\eps\int_{C_\eps}\nabla u^\eps\cdot \nabla\fhi^\eps
      =\frac{1}{\eps^2 L}\int_{C_\eps\cap\{x_1\in I^\eps_{c,d}\}} \del_2 u^\eps
      = \int_{\{x_1\in I^\eps_{c,d}\}} j^\eps \,d\calL^2\,.
  \end{align*}
  We claim that this implies, as $\eps\to 0$, 
  \begin{equation}\label{eq:claim1-Beps}
    B_\eps \to \int_{[c,d]\times\{0\}}\,dj_*\,.
  \end{equation}
  Indeed, for every $\delta>0$ and for every $\eps< \delta$,
  there holds, as $\eps\to 0$,
  \begin{align*}
    &\left| B_\eps - \int_{\{x_1 \in [c,d] \}}\, dj_* \right|
      \le \left| \int_{\{x_1\in (0,a) \cap (c-\delta , d+\delta)\} } j^\eps \, d\calL^2
      - \int_{\{x_1\in (c-\delta , d+\delta)\} }\, dj_* \right|\\
    &\qquad\qquad
      +  \int_{\{x_1\in (c-\delta,c)\cup (d,d+\delta)\} } J^\eps\, d\calL^2
      + \left| \int_{\{x_1\in (c-\delta,c)\cup (d,d+\delta)\} } dj_* \right|\\
    &\quad \to
      \int_{\{x_1\in (c,c+\delta)\cup (d-\delta,d)\} } dJ_*
      + \left| \int_{\{x_1\in (c,c+\delta)\cup (d-\delta,d)\} } dj_* \right|\,.
  \end{align*}
  By outer regularity of the Radon measures $J_*$ and $j_*$, the right
  hand side is arbitrarily small for small $\delta>0$. This verifies
  \eqref {eq:claim1-Beps}.

  \smallskip \textit{Second calculation of $B_\eps$.} The second
  calculation of $B_\eps$ is based on a quite elementary integration
  by parts: The integral of the derivative is given by the difference
  of values at top and bottom of the channels.

  To perform the calculation, we need some additional notation.
  Recall that the microscopic channels are defined as
  $C_\eps := \bigcup_{k_1 = 0}^{a/\eps - 1} (k_1 \eps, k_1 \eps +
  \alpha\eps^3)\times [0,L\eps]$. We define the union of the lower and
  upper channel boundaries in the interval $(c,d)$ as
  \begin{equation*}    
    \Gamma_\eps^U := \bigcup_{k_1\in \calK_\eps}
    (k_1\eps, k_1\eps + \alpha\eps^3)\times \{0\}
    \quad\text{ and }\quad
    \Gamma_\eps^V := \bigcup_{k_1\in \calK_\eps}
    (k_1\eps, k_1\eps + \alpha\eps^3)\times \{\eps L\}\,.
  \end{equation*}
  With this notation, an integration by parts provides
  \begin{align}
    \begin{split}
      \label{eq:Bepsfinal} B_\eps
      =\frac{1}{\eps^2 L}\int_{C_\eps\cap\{x_1\in I^\eps_{c,d}\}}\del_2 u^\eps
      =\frac{1}{\eps^2 L}\left(\int_{\Gamma_\eps^V}u^\eps-\int_{\Gamma_\eps^U}u^\eps\right)\,.
    \end{split}
  \end{align}
  It remains to determine the limit on the right hand side
  of \eqref{eq:Bepsfinal}. We will prove that
  \begin{align}
    \label{eq:meanvalueschannels1}
    &\lim_{\eps\to 0}\,\frac{1}{\eps^2} \int_{\Gamma_\eps^V}u^\eps
      = \alpha \int_{c}^d v(x_1)\,dx_1\,,\\
    \label{eq:meanvalueschannels2}
    &\lim_{\eps\to 0}\,\frac{1}{\eps^2}\int_{\Gamma_\eps^U}u^\eps
      = \alpha \int_{c}^d u(x_1,0)\,dx_1\,.
  \end{align}
  Once \eqref{eq:meanvalueschannels1}--\eqref {eq:meanvalueschannels2}
  is shown, the proof of the proposition is complete: together with
  \eqref {eq:claim1-Beps}, we obtain
  \begin{equation*}
    \int_{[c,d] \times\{0\}}\,dj_*
    =\lim_{\eps\to 0} B_\eps
    =\frac{\alpha}{L}\int_{c}^d(v(x_1)-u(x_1,0))\,dx_1\,.
  \end{equation*}
  Since $[c,d]$ was arbitrary, relations
  \eqref{eq:reljstarv}--\eqref{eq:geometricflowrule} are verified.

  \smallskip
  \textit{Verification of \eqref{eq:meanvalueschannels1}--\eqref
    {eq:meanvalueschannels2}.}  We consider the unit cell
  $Y:=(0,1)\times \left(-1,L+V\right)$ and the index set $\calK_\eps$
  of \eqref{eq:defIeps} and study the following averaged functions on
  $Y$:
  \begin{align*}
    U^\eps(y_1,y_2):=\frac{1}{|\calK_\eps|}
    \sum_{k_1\in \calK_\eps}u^\eps(\eps(k_1+y_1), \eps y_2)\quad
    \text{for }(y_1,y_2)\in Y\,.
  \end{align*}

  The (rescaled) channel in the periodicity cell $Y$ is
  $C^\eps_Y := (0,\alpha\eps^2)\times [0,L] \subset Y$. Its lower and
  upper boundary are the sets
  $\Gamma_Y^{\eps,U} := [0,\alpha\eps^2]\times\{0\}$ and
  $\Gamma_Y^{\eps,V}:= [0,\alpha\eps^2]\times\{L\}$.  The domain below
  the channel is $Y_U := (0,1)\times (-1,0)$, the domain above the
  channel is $Y_V := (0,1)\times (L,L+V)$.

  Our first aim is to prove that the restrictions $U^\eps|_{Y_U}$ and
  $U^\eps|_{Y_V}$ converge (weakly in $L^2$) to constant functions.
  
  Using Jensen's inequality and the fact that
  $|\calK_\eps|^{-1} = \eps/(d-c) + O(\eps^2)$, we find
  \begin{align*}
    \int_{Y_U}|U^\eps(y)|^2\,dy
    &\leq
      \frac{1}{|\calK_\eps|}\sum_{k_1\in \calK_\eps}
      \int_{Y_U}|u^\eps(\eps(y_1+k_1),\eps y_2)|^2\,dy\\
    &= \frac{1}{\eps^2|\calK_\eps|}\sum_{k_1\in
      \calK_\eps}\int_{\left(k_1\eps,k_1\eps + \eps \right)\times
      \left(-\eps,0\right)}|u^\eps(x)|^2\,dx\\
    &\leq \frac{1}{\eps^2}\left(\frac{\eps}{d-c}+O(\eps^2)\right)\int_{(c,d)\times
      (-\eps,0)}|u^\eps(x)|^2\,dx\\
    &\leq \left(\frac{1}{d-c}+O(\eps)\right)\left(\frac1\eps\int_{(c,d)\times
      (-\eps,0)}|u^\eps(x)|^2\,dx\right)\leq C\,,
  \end{align*}
  where in the last step we exploited the boundedness of $u^\eps$ in
  $H^1(\Omega_0)$: The second bracket in the last line converges to
  the $L^2$-norm of the trace of $u^\eps$ on $\Gamma_0$. We have
  obtained that the sequence $U^\eps|_{Y_U}$ is uniformly bounded in
  $L^2(Y_U)$.

  Regarding $U^\eps|_{Y_V}$ we perform the same calculation and use,
  in the last step,
  \begin{align*}
    &\int_{Y_V}|U^\eps(y)|^2\,dy
    \leq C \frac1\eps\int_c^d \int_{L\eps}^{(L+V)\eps} |u^\eps(x)|^2\,dx\\
    &\qquad \leq 2C \frac1\eps\int_c^d \int_{L\eps}^{(L+V)\eps} |u^\eps(x) - v^\eps(x_1)|^2\,dx
      + 2C \frac1\eps\int_c^d \int_{L\eps}^{(L+V)\eps} |v^\eps(x_1)|^2\,dx\\
    &\qquad \leq O(\eps) + 2C V \int_c^d |v^\eps(x_1)|^2\,dx_1 \le C\,,
  \end{align*}
  where we used the one-dimensional Poincar\'e (also called
  Poincar\'e-Wirtinger) inequality with averages $v^\eps(x_1)$ for the
  first integral, exploiting that the domain size is $V\eps$. In the
  last estimate we used the boundedness of $v^\eps$ in $L^2((0,a))$
  that was assumed in Assumption \ref {ass:v-w}.

  The above estimates allow to proceed with the weak $L^2$-compactness
  of bounded sequences.  There exist limit functions $U$ and $V$ such
  that, up to a subsequence, $U^\eps|_{Y_U}\weakto U$ in
  $L^2\left(Y_U\right)$ and $U^\eps|_{Y_V}\weak V$ in
  $L^2\left(Y_V\right)$ as $\eps\to 0$. It is not difficult to verify
  that the limit functions $U$ and $V$ are constant functions. Indeed,
  the gradient of $U^\eps|_{Y_U}$ satisfies, in the limit $\eps\to 0$,
  \begin{align*}
    \int_{Y_U}|\nabla U^\eps(y)|^2\,dy
    &\leq\frac{\eps^2}{|\calK_\eps|}\sum_{k_1\in \calK_\eps} \int_{Y_U}|\nabla
                   u^\eps(\eps(y_1+k_1),\eps y_2)|^2\,dy\\
    &= \frac{1}{|\calK_\eps|}\sum_{k_1\in
         \calK_\eps}\int_{\left(k_1\eps,k_1\eps + \eps\right)\times
          \left(-\eps,0\right)}|\nabla u^\eps(x)|^2\,dx\\
    &\leq  \left(\frac{\eps}{d-c}+O(\eps^2)\right)\int_{(c,d)\times
        (-\eps,0)}|\nabla u^\eps(x)|^2\,dx\to 0\,,
  \end{align*}
  since $u^\eps$ is bounded in $H^1(\Omega_\eps)$. Analogously,
  $\|\nabla U^\eps\|_{L^2\left(Y_V\right)}\to 0$ and we obtain
  $\nabla U = \nabla V = 0$. As a consequence, for two real numbers
  $\xi_U, \xi_V \in \R$, the constant functions are $U\equiv \xi_U$
  and $V\equiv \xi_V$.

  In our next step we identify the constants $\xi_U$ and
  $\xi_V$. There holds
  \begin{align*}
    \xi_U\leftarrow
    &\int_{Y_U}U^\eps(y)\,dy=\frac{1}{|\calK_\eps|}\sum_{k_1\in \calK_\eps}
      \int_{Y_U} u^\eps(\eps(y_1+k_1,\eps y_2)\,dy\\
    =&\frac{1}{\eps^2|\calK_\eps|}\sum_{k_1\in \calK_\eps}
       \int_{\left(k_1\eps,k_1\eps + \eps\right)\times \left(-\eps,0\right)}u^\eps(x_1,x_2)\,dx\\
    =&\left(\frac{1}{d-c}+O(\eps)\right)\left(\frac{1}{\eps}
       \int_{I^\eps_{(c,d)}\times(-\eps,0)}u^\eps(x_1,x_2)\,dx\right)\\
    \rightarrow\,&\frac{1}{d-c}\int_{c}^d u(x_1,0)\,dx_1\,.
  \end{align*}
  Analogously, using definition \eqref{eq:v-eps} of $v^\eps$ and the
  weak convergence $v_\eps\weakto v$ in $L^2(I)$,
  \begin{align*}
    V\, \xi_V
    &\leftarrow \int_{Y_V}U^\eps(y)\,dy=\frac{1}{|\calK_\eps|}\sum_{k_1\in \calK_\eps}
      \int_{Y_V} u^\eps(\eps(y_1+k_1,\eps y_2)\,dy\\
    &=\left(\frac{1}{d-c}+O(\eps)\right)V\left(\frac{1}{V\eps}
      \int_{I^\eps_{(c,d)}\times(\eps L,\eps(L+V))}u^\eps(x)\,dx\right)\\
    &=\left(\frac{V}{d-c}+O(\eps)\right)\int_{I^\eps_{(c,d)}}v^\eps(x_1)\,dx_1\\
    &\rightarrow \frac{V}{d-c}\int_c^dv(x_1)\,dx_1\,.
  \end{align*}
  We have found
  \begin{equation}
    \label{eq:identifyxi}
    \xi_U = \frac{1}{d-c}\int_{c}^d u(x_1,0)\,dx_1\quad\text{and}\quad
    \xi_V = \frac{1}{d-c}\int_c^d v(x_1)\,dx_1\,.
  \end{equation}

  At this point, we identified the averages of $U^\eps$ (below and
  above the channel) with $u$ and $v$. In order to check \eqref
  {eq:meanvalueschannels1}--\eqref {eq:meanvalueschannels2}, it
  remains to relate averages of $U^\eps$ in the bulk areas to averages
  of $U^\eps$ in the ends of the channel. This can be done with a
  Lemma that was proved and used in \cite
  {Lamacz-Schweizer-manyres-2017}.

  We use Lemma A.1 of \cite {Lamacz-Schweizer-manyres-2017} with
  slightly adapted notation. The obstacle in the single cell $Y$ is
  given by
  $\Sigma_Y^\eps := Y\setminus (\overline{Y_U} \cup\overline{Y_V}
  \cup\overline{C_Y^\eps})$. We furthermore assume only boundedness of
  the gradients of $U^\eps$ and not the $L^2$-boundedness everywhere
  (also in the channel). An inspection of the proof in \cite
  {Lamacz-Schweizer-manyres-2017} shows that this is sufficient.

  The essential part of the proof is the following: With the
  tangential vector in each channel being $e_2$, one considers the
  functions $V^\eps := \del_2 U^\eps$. These functions solve the same
  Helmholtz equation and they satisfy homogeneous boundary conditions:
  Dirichlet conditions on one part of the boundary, Neumann conditions
  on the other. This allows to multiply the equation for $V^\eps$ with
  $V^\eps$. One finds uniform $H^1$-estimates for $V^\eps$ which yield
  uniform $H^2$-estimates for $U^\eps$.  The embedding
  $H^2(Y_U) \subset C^0(Y_U)$ (accordingly for $Y_V$) allows to
  compare point values of $U^\eps$ with averages of $U^\eps$.

  \begin{lemma}[Adaption of Lemma A.1 from \cite
    {Lamacz-Schweizer-manyres-2017}]
    Let $U^\eps:Y\setminus \overline{\Sigma_Y^\eps}\to\R$ be a family
    of $H^1$-functions such that the $L^2$-norms of $\nabla U^\eps$
    are bounded. We assume that every $U^\eps$ solves the Helmholtz
    equation
    \begin{align*}
      -\Delta U^\eps&=\omega^2\eps^2 U^\eps\quad
                      \text{in } Y\setminus \overline{\Sigma_Y^\eps}\,,\\
      \del_n U^\eps&=0\quad\text{on } \del\Sigma_Y^\eps\,. 
    \end{align*}
    No boundary conditions are imposed on $\del Y$. Assume that
    \begin{align*}
      &U^\eps|_{Y_U}\weakto\xi_U\quad\text{in }L^2(Y_U)\,,\\ 
      &U^\eps|_{Y_V}\weakto\xi_V\quad\text{in }L^2(Y_V)\,,
    \end{align*}
    as $\eps\to 0$. Then
    \begin{align}
      \label{eq:resultA1}
      \mean_{\Gamma_Y^{\eps,U}}U^\eps(y)\,d\mathcal{H}^1(y)\to \xi_U\quad\text{and}
      \quad \mean_{\Gamma_Y^{\eps,V}}V_\eps(y)\,d\mathcal{H}^1(y)\to \xi_V\,.
    \end{align}
  \end{lemma}

  The Lemma can indeed be applied. (a) $U^\eps$ solves the (rescaled)
  Helmholtz equation with Neumann boundary condition since $u^\eps$
  satisfies the (non-rescaled) system.  (b) The $L^2$-boundedness of
  $\nabla U^\eps$ follows easily from $H^1$-boundedness of $u^\eps$
  (compare the calculations above in this proof). (c) The weak $L^2$
  limits $\xi_U$ and $\xi_V$ have been verified above.

  \smallskip With the result of the lemma at hand, it only remains to
  compare the limits in \eqref {eq:meanvalueschannels1}--\eqref
  {eq:meanvalueschannels2} with the limits in \eqref{eq:resultA1}. We
  calculate with
  $\Gamma_Y^{\eps,U}:=\overline{C^\eps_Y}\cap
  \overline{Y_U}=[0,\alpha\eps^2]\times\{0\}$:
  \begin{align*}
    &\mean_{\Gamma_Y^{\eps,U}}U^\eps(y)\,d\mathcal{H}^1(y)
    = \frac{1}{\alpha\eps^2} \int_{\Gamma_Y^{\eps,U}}\frac{1}{|\calK_\eps|}
      \sum_{k_1\in \calK_\eps}u^\eps(\eps(y_1+k_1,\eps y_2)\,d\mathcal{H}^1(y)\\
    &\qquad
      =\frac{1}{\alpha\eps^2}\frac{1}{|\calK_\eps|}
      \frac1\eps\int_{\Gamma_\eps^{U}}u^\eps(x)\,d\mathcal{H}^1(x)
      =\left(\frac{1}{\eps^2\alpha(d-c)} + O\left(\frac{1}{\eps}\right)\right)
      \int_{\Gamma_\eps^{U}}u^\eps(x)\,d\mathcal{H}^1(x)\,.
  \end{align*}
  The results \eqref{eq:resultA1} and \eqref{eq:identifyxi} thus imply
  \begin{align*}
    \frac{1}{\eps^2}\int_{\Gamma_\eps^U}u^\eps(x)\,d\mathcal{H}^1(x)
    &\to \alpha (d-c) \xi_U = \alpha \int_c^d u(x_1,0)\,dx_1\,,
  \end{align*}
  which is the claim \eqref{eq:meanvalueschannels2}. The limit
  \eqref{eq:meanvalueschannels1} is obtained in an analogous way.
\end{proof}

It remains to derive a further relation between $u$, $v$, and $j$. We
obtain a relation from mass conservation in the resonator volume: The
flux through the channels (and hence the density $j$) can be expressed
in terms of $v$.

\begin{proposition}[Mass conservation]
  \label{prop:B}
  Let $j$ be as in Proposition \ref{prop:geometricflowrule} and $v$ as
  in Assumption \ref{ass:v-w}. Then
  \begin{equation}
    \label{eq:densityjstar-distribution}
    j = V \del_1^2 v + V \omega^2\, v
  \end{equation}
  in the sense of distributions. Furthermore, for every
  $\psi\in C^1(\bar I)$,
  \begin{equation}
    \label{eq:densityjstar1}
    \int_I j(x_1) \psi(x_1)\, dx_1 =
    - \int_I V \del_1 v(x_1) \del_1 \psi(x_1)\, dx_1
    + \int_I V \omega^2 v(x_1) \psi(x_1) \, dx_1\,,
  \end{equation}
  and $v$ has the regularity $\del^2_1 v\in L^1(I)$.
\end{proposition}

\begin{proof}
  We fix $\psi\in C^2(\bar I) = C^2([0,a])$ with
  $\del_1 \psi(0) = \del_1 \psi(a) = 0$. We use $\theta^\eps$ of
  \eqref{eq:functionlargegradient} and consider
  $\fhi^\eps(x_1,x_2) := \psi(x_1) \theta^\eps(x_2)$ in equation
  \eqref{eq:weakformweps} for $w^\eps$.  The left hand side vanishes
  since $\fhi^\eps$ vanishes in $\Omega_0$. There remains
  \begin{equation}
    \label{eq:weakformsimple}
    \frac{1}{\eps} \int_{S_\eps \cup C_\eps} \nabla u^\eps\cdot \nabla\fhi^\eps
    =\frac{\omega^2}{\eps} \int_{S_\eps \cup C_\eps}  u^\eps\, \fhi^\eps\,.
  \end{equation}
  
  {\em Left hand side of \eqref{eq:weakformsimple}.} We calculate
  \begin{align*}
    \frac{1}{\eps} \int_{S_\eps \cup C_\eps} \nabla u^\eps\cdot \nabla\fhi^\eps
    &= \frac{1}{\eps} \int_{ C_\eps} \del_1 u^\eps(x_1,x_2) \del_1\psi(x_1)\,\frac{x_2}{\eps L}
      + \frac{1}{\eps} \int_{S_\eps} \del_1 u^\eps(x_1,x_2) \del_1\psi(x_1)\\
    &\quad + \frac{1}{\eps} \int_{ C_\eps}\del_2 u^\eps\frac{1}{L\eps}\psi(x_1)\,.
  \end{align*}
  Since $\del_1 u^\eps$ satisfies \eqref{eq:assdel1ueps}, the first
  integral vanishes in the limit $\eps\to 0$. For the second term we
  find
  \begin{equation*}
    \frac{1}{\eps} \int_{S_\eps} \del_1 u^\eps(x_1,x_2) \del_1\psi(x_1)
    =-\frac{1}{\eps}\int_{S_\eps} u^\eps(x_1,x_2) \del^2_1\psi(x_1)
    \to- V\int_I v(x_1)\del^2_1\psi(x_1)\,dx_1
  \end{equation*}
  by the weak convergence $v^\eps \weakto v$ in $L^1(I)$.  For the
  last term we exploit Lemma \ref{lem:currentintroduction}, which
  ensures that
  \begin{equation*}
    \frac{1}{\eps} \int_{ C_\eps}\del_2 u^\eps\frac{1}{L\eps}\psi(x_1)
    =\int_{\R^2}j^\eps\psi(x_1)\to \int_{\Gamma_0}\psi(x_1)\,dj_*(x)
    =\int_I \psi(x_1)j(x_1)\,dx_1\,. 
  \end{equation*}
  
  {\em Right hand side of \eqref{eq:weakformsimple}.} We obtain
  \begin{align*}
    \frac{\omega^2}{\eps} \int_{S_\eps \cup C_\eps}  u^\eps\, \fhi^\eps
    &=\frac{\omega^2}{\eps}
      \left(\int_{C_\eps}  u^\eps(x_1,x_2)\, \psi(x_1)\frac{x_2}{\eps L} + 
      \int_{S_\eps}  u^\eps(x_1,x_2)\, \psi(x_1)\right)\\
    &\to V\omega^2\int_{I}v(x_1)\psi(x_1) 
  \end{align*}
  as $\eps\to 0$, where we used the convergence of averages in
  $S_\eps$ to $v$ and, for the first term,
  $\eps^{-1} \int_{C_\eps} | u^\eps | \le \eps^{-1} \| u^\eps \|_{L^2}
  |C_\eps|^{1/2} \le C \eps^{1/2} \to 0$.  We obtain from
  \eqref{eq:weakformsimple}
  \begin{equation}
    \label{eq:densityjstar1-267}
    \int_I j(x_1) \psi(x_1)\, dx_1 =
     \int_I V v(x_1) \del_1^2 \psi(x_1)\, dx_1
    + \int_I V \omega^2 v(x_1) \psi(x_1) \, dx_1\,.
  \end{equation}

  Relation \eqref {eq:densityjstar1-267} provides \eqref
  {eq:densityjstar-distribution}. In particular, the distribution
  $\del_1^2 v$ is expressed by the $L^1$-functions $v$ and $j$
  (compare Proposition \ref{prop:geometricflowrule}). We therefore
  find $v\in W^{2,1}(I)$.

  Relation \eqref {eq:densityjstar1} follows with another integration
  by parts from \eqref {eq:densityjstar1-267}. The set of test
  functions is dense in $H^1(I)$, hence \eqref {eq:densityjstar1}
  holds for all $\psi\in H^1(I)$ and, in particular, for all
  $\psi\in C^1(\bar I)$.
\end{proof}

We can now formally conclude the proof of our main theorem.

\begin{proof}[Proof of Theorem \ref{thm:mainTheorem}.] Equation \eqref
  {eq:main-w-weak} for $w$ was checked in Proposition \ref{prop:A}.
  Due to Propositions \ref{prop:B} and \ref{prop:geometricflowrule} we
  find
  \begin{align*}
    V(\del^2_1 v(x_1)+\omega^2 v(x_1))=j(x_1)=\frac{\alpha}{L}(v(x_1)-u(x_1,0))
  \end{align*}
  in the sense of distributions.  Re-ordering terms, we may write this
  relation equivalently as
  \begin{align*}
    -\del_1^2 v(x_1)+\left(\frac{\alpha}{LV}-\omega^2\right)v(x_1)
      =\frac{\alpha}{LV}u(x_1,0)\,,
  \end{align*} 
  which is relation \eqref{eq:limiteq-v}.  For a test function
  $\psi\in C^1(\bar I)$, the distribution $\del_1^2 v$ can be
  integrated by parts once without boundary terms. We therefore have
  obtained also \eqref{eq:limiteq-v-weak}, which encodes additionally
  the homogeneous Neumann boundary condition for $v$.
\end{proof}

\subsection*{Acknowledgement} This work was initiated when the three
authors attended in summer 2019 workshop 1931 ``Computational
Multiscale Methods'' in Oberwolfach. The invitation and the
hospitality of the institute are gratefully acknowledged.

\bibliographystyle{abbrv}
\bibliography{lit-Helmholtz}

\end{document}